\documentclass[11pt,psamsfonts,twoside]{article}
\usepackage{amsfonts}
\usepackage{amsfonts,latexsym,bm}
\usepackage{amsmath,amsthm}
\usepackage{amssymb,amscd}
\usepackage{amsfonts,amsbsy}
\usepackage{fancyhdr,graphicx}
\usepackage[dvips]{psfrag}
\usepackage{indentfirst}

\numberwithin{equation}{section}
\newtheorem {theorem} {Theorem}[section]
\newtheorem {proposition} [theorem]{Proposition}
\newtheorem {corollary} [theorem]{Corollary}
\newtheorem {definition} [theorem]{Definition}
\newtheorem {lemma}  [theorem]{Lemma}

\newcommand{\supp}{{\rm supp}}

\begin{document}
\setlength{\parindent}{4ex} \setlength{\parskip}{1ex}
\setlength{\oddsidemargin}{12mm} \setlength{\evensidemargin}{9mm}

\title{{Existence of Solutions for the Debye-H\"{u}ckel System with Low Regularity
Initial Data}}
\author{Jihong Zhao$^{1}$, \ Qiao Liu$^{2}$,  \ and\ Shangbin Cui$^{2}$\footnote{E-mail: zhaojihong2007@yahoo.com.cn;
liuqao2005@163.com; cuisb3@yahoo.com.cn.}\\
[0.2cm] {\small $^{1}$ Department of Mathematics, Northwest A\&F
University, Yangling,}\\
{\small Shaanxi 712100, People's Republic of China}\\
[0.2cm] {\small $^{2}$ Department of Mathematics, Sun Yat-sen
University, Guangzhou, }\\
{\small Guangdong 510275, People's Republic of China}}
\date{}
\maketitle

\begin{abstract}
In this paper we study existence of solutions for the Cauchy problem
of the Debye-H\"{u}ckel system with low regularity initial data. By
using the Chemin-Lerner time-space estimate for the heat equation,
we prove that there exists a unique local solution if the initial
data belongs to the Besov space $\dot{B}^{s}_{p,q}(\mathbb{R}^{n})$
for $-\frac{3}{2}<s\leq-2+\frac{n}{2}$, $p=\frac{n}{s+2}$ and $1\leq
q\leq \infty$, and furthermore, if the initial data is sufficiently
small then the solution is global. This result improves the
regularity index of the initial data space in previous results on
this model.


\end{abstract}

\section{\bf Introduction}

In this paper we study the following Debye-H\"{u}ckel system arising
from the theory of electrolytes (\cite{DH23}):
\begin{equation}\label{eq1.1}
\begin{cases}
  \partial_{t}v=\nabla\cdot(\nabla v-v\nabla \phi)\quad &\mbox{in}\ \
  \mathbb{R}^n\times(0,\infty),\\
  \partial_{t}w=\nabla\cdot(\nabla w+w\nabla\phi)\quad &\mbox{in}\ \
  \mathbb{R}^n\times(0,\infty),\\
  \Delta \phi=v-w\quad &\mbox{in}\ \
  \mathbb{R}^n\times(0,\infty),\\
  v(x,0)=v_0(x),\quad
  w(x,0)=w_0(x) \ \  &\mbox{in}\ \ \mathbb{R}^n,
\end{cases}
\end{equation}
where $v$ and $w$ denote the densities of the electron and the hole,
respectively, in an electrolytes, and $\phi$ denotes the electric
potential.

Mathematical analysis of the Debye-H\"{u}ckel system was first
focused on the initial boundary value problems in 1980's, and some
results related to the global existence, uniqueness and regularity
of classical solutions and the asymptotic stability of stationary
solutions were obtained by using the Green function, the
Poincar\'{e} inequality and the standard maximum principle of
equations of parabolic type; see \cite{G85}, \cite{GG86}, \cite{M74}
and \cite{S83} for more details. In 1994, Biler, Hebisch and
Nadzieja in \cite{BHN94} considered the no-flux boundary problem of
\eqref{eq1.1}, and considered global and local existence of weak
solutions, convergence rate estimates to stationary solutions of
time-dependent solutions, and for further studies related to this
topic we refer the reader to see \cite{BMV04}, \cite{BD00} and the
references therein.

In 1999,  Karch in \cite{K99} proved existence and uniqueness of
solutions of the problem \eqref{eq1.1} for initial data in the Besov
space $\dot{B}^{s}_{p, \infty}(\mathbb{R}^n)$ with $-1<s<0$ and
$p=\frac{n}{s+2}$. Note that similar results for initial data in the
Lebesgue and Sobolev spaces was established only recently, see the
work of Kurokiba and Ogawa \cite{KO08}. In \cite{OS08}, Ogawa and
Shimizu considered existence of global solutions of the problem
\eqref{eq1.1} for small initial data in a two-dimensional critical
Hardy space. The purpose of this paper is to prove existence of
solutions for \eqref{eq1.1} with initial data in the Besov space
$\dot{B}^{s}_{p,q}(\mathbb{R}^{n})$ of indices
$-\frac{3}{2}<s\leq-2+\frac{n}{2}$, $p=\frac{n}{s+2}$ and $1\leq
q\leq \infty$. This result improves the corresponding result of
Karch obtained in \cite{K99}. It shows that the Debye-H\"{u}ckel
system has a better property than the Navier-Stokes equations in
regard to existence of solutions, since for that equation there is
no existence result for initial data in a space with regularity
index $s$ smaller than $-1$. In fact, the nonlinear term of the
system \eqref{eq1.1} seems to be closer to the quadratic nonlinear
heat equation ($\sim u^2$) than to the Navier-Stokes equations
($\sim u\cdot \nabla u$), and our main result (see Theorem
\ref{th1.1} below) also holds  for the quadratic nonlinear heat
equation and is even new for this equation. Main tools used to get
our main result are the Chemin-Lerner space $\mathfrak{L}^{r}(0,T;
\dot{B}^{-2+n/p+2/r}_{p,q})$ and some related estimates (see
Definition 2.1 and Propositions 2.2 and 2.3 in Section 2).

Note that from the third equation in (1.1) we have
\begin{equation}\label{eq1.2}
  \phi=(-\Delta)^{-1}(w-v):=E\ast(w-v),
\end{equation}
where $E(x)=-\frac{1}{2\pi}\log|x|$ for $n=2$ and
$E(x)=\frac{1}{4}\pi^{-n/2}\Gamma(\frac{n}{2}-1)|x|^{-(n-2)}$ for
$n\ge 3$, so that we can eliminate $\phi$ from \eqref{eq1.1} and
obtain
\begin{equation}\label{eq1.3}
\begin{cases}
  \partial_{t}v-\Delta v=-\nabla\cdot(v\nabla(-\Delta)^{-1}(w-v))\quad &\mbox{in}\ \
  \mathbb{R}^n\times(0,\infty),\\
  \partial_{t}w-\Delta w=\nabla\cdot(w\nabla(-\Delta)^{-1}(w-v))\quad &\mbox{in}\ \
  \mathbb{R}^n\times(0,\infty),\\
  v(x,0)=v_0(x),\quad
  w(x,0)=w_0(x) \ \  &\mbox{in}\ \ \mathbb{R}^n.
\end{cases}
\end{equation}
Hence, we only need to consider this equivalent problem. We now give
the precise statement of our main result, and  for simplicity, we
use $(v, w)\in \mathcal{X}$ to denote  $(v, w)\in
\mathcal{X}\times\mathcal{X}$ for a Banach space $\mathcal{X}$.

\begin{theorem}\label{th1.1}
Let $n\geq2$, $1\leq q\leq \infty$ and $2\leq p<2n$. Suppose that
$(v_{0}, w_{0})\in \dot{B}^{-2+n/p}_{p,q}(\mathbb{R}^{n})$. There
exists $T>0$ such that the problem \eqref{eq1.1} has a unique
solution
\begin{equation}\label{eq1.4}
  (v, w)\in \underset{1<r\leq\infty}{\cap}\mathfrak{L}^{r}(0, T;
  \dot{B}^{-2+n/p+2/r}_{p,q}(\mathbb{R}^{n})).
\end{equation}
 Moreover, if
$(v_{0},w_{0})$ belongs to the closure of
$\mathcal{S}(\mathbb{R}^n)$ in
$\dot{B}^{-2+n/p}_{p,q}(\mathbb{R}^{n})$, then $(v, w)\in C([0,T),
\dot{B}^{-2+n/p}_{p,q}(\mathbb{R}^{n}))$. Besides, there exists
$\varepsilon>0$ such that if $\|(v_{0},
w_{0})\|_{\dot{B}^{-2+n/p}_{p,q}}\leq\varepsilon$, then the above
assertion holds for $T=\infty$, i.e., the solution $(v, w)$ is
global. Furthermore, if $(v,w)$ and $(\tilde{v}, \tilde{w})$ are two
solutions of \eqref{eq1.1} with initial data $(v_0, w_0)$ and
$(\tilde{v}_0, \tilde{w}_0)$, respectively, then there exists a
universal constant $C>0$ such that for any $1<r\leq\infty$, we have
\begin{equation}\label{eq1.5}
  \|(v-\tilde{v}, w-\tilde{w})\|_{\mathfrak{L}^{r}(0,T; \dot{B}^{-2+n/p+2/r}_{p,q})}\le
  C\|(v_0-\tilde{v}_0,
  w_0-\tilde{w}_0)\|_{\dot{B}^{-2+n/p}_{p,q}}.
\end{equation}
\end{theorem}


It is easy to verify that \eqref{eq1.1} is invariant under the
scaling $v_{\lambda}(x,t)=\lambda^2 v(\lambda x, \lambda^2t)$,
$w_{\lambda}(x,t)=\lambda^2 w(\lambda x, \lambda^2t)$ and
$\phi_{\lambda}(x,t)=\phi(\lambda x, \lambda^2t)$. Hence, as a
standard practice, we have the following existence result for {\em
self-similar solution} of \eqref{eq1.1}:

\begin{corollary}\label{co1.2}
Let $n\ge2$ and $\frac{n}{2}\leq p<2n$. Suppose that $(v_{0},
w_{0})\in \dot{B}^{-2+n/p}_{p,\infty}(\mathbb{R}^{n})$ and
$\|(v_{0}, w_{0})\|_{\dot{B}^{-2+n/p}_{p,q}}\leq\varepsilon$, where
$\varepsilon$ is as above. Suppose furthermore that $v_{0}, w_{0}$
are homogeneous of degree $-2$, i.e., they satisfy the relations
$v_{0}(x)=\lambda^{2}v_{0}(\lambda x)$ and $w_{0}(x)
=\lambda^{2}w_{0}(\lambda x)$ for all $x\in\mathbb{R}^{n}$ and
$\lambda>0$. Then the unique global solution ensured by Theorem
\ref{th1.1} is a self-similar solution, i.e., it satisfies the
following condition:
\begin{equation*}
  v(x,t)=\lambda^2 v(\lambda x, \lambda^2t),\ \
  w(x,t)=\lambda^2 w(\lambda x, \lambda^2t),\ \ \phi(x,t)=\phi(\lambda x,
  \lambda^2t).
\end{equation*}
\end{corollary}

In the next section we give the proof of Theorem 1.1.

\section{The proof of Theorem 1.1}

We first recall some basic notions and preliminary results used in
the proof of Theorem 1.1. Let $\mathcal{S}(\mathbb{R}^{n})$ be the
Schwartz space and $\mathcal{S}'(\mathbb{R}^{n})$ be its dual. Given
$f\in\mathcal{S}(\mathbb{R}^{n})$, the Fourier transform of it,
$\mathcal{F}(f)=\widehat{f}$, is defined by
$$
  \mathcal{F}(f)(\xi)=\widehat{f}(\xi)=\frac{1}{(2\pi)^{n/2}}\int_{\mathbb{R}^{n}}f(x)e^{-ix\cdot\xi}dx.
$$
Let $\mathcal{D}_{1}=\{\xi\in\mathbb{R}^{n},\
|\xi|\leq\frac{4}{3}\}$ and
$\mathcal{D}_{2}=\{\xi\in\mathbb{R}^{n},\ \frac{3}{4}\leq
|\xi|\leq\frac{8}{3}\}$. Choose two non-negative functions $\phi,
\psi\in\mathcal{S}(\mathbb{R}^{n})$ supported, respectively, in
$\mathcal{D}_{1}$ and $\mathcal{D}_{2}$ such that
\begin{align*}
  \psi(\xi)+\sum_{j\geq0}\phi(2^{-j}\xi)=1, \ \ \xi\in\mathbb{R}^{n},\\
  \sum_{j\in\mathbb{Z}}\phi(2^{-j}\xi)=1, \ \
  \xi\in\mathbb{R}^{n}\backslash\{0\}.
\end{align*}
We denote $\phi_{j}(\xi)=\phi(2^{-j}\xi)$, $h=\mathcal{F}^{-1}\phi$
and $\tilde{h}=\mathcal{F}^{-1}\psi$, where $\mathcal{F}^{-1}$ is
the inverse Fourier transform. Then the dyadic blocks $\Delta_{j}$
and $S_{j}$ can be defined as follows:
\begin{align*}
  \Delta_{j}f=\phi(2^{-j}D)f=2^{jn}\int_{\mathbb{R}^{n}}h(2^{j}y)f(x-y)dy,\\
  S_{j}f=\psi(2^{-j}D)f=2^{jn}\int_{\mathbb{R}^{n}}\tilde{h}(2^{j}y)f(x-y)dy.
\end{align*}
Here $D=(D_1, D_2, \cdots, D_n)$ and $D_j=i^{-1}\partial_{x_j}$
($i^{2}=-1$). The set $\{\Delta_{j}, S_{j}\}_{j\in\mathbb{Z}}$ is
called the Littlewood-Paley decomposition. Formally,
$\Delta_{j}=S_{j}-S_{j-1}$ is a frequency projection to the annulus
$\{|\xi|\sim 2^{j}\}$, and $S_{j}=\sum_{k\leq j-1}\Delta_{k}$ is a
frequency projection to the ball $\{|\xi|\leq 2^{j}\}$. For more
details, please reader to \cite{C98} and \cite{L02}.  Let
$\mathcal{Z}(\mathbb{R}^{n})=\big\{f\in \mathcal{S}(\mathbb{R}^{n}):
\ \ \partial^{\alpha}\widehat{f}(0)=0, \
\forall\alpha\in(\mathbb{N}\cup\{0\})^{n}\big\}$, and denote by
$\mathcal{Z}'(\mathbb{R}^{n})$ the dual of it. Recall that for
$s\in\mathbb{R}$ and $(p,q)\in[1, \infty]\times[1, \infty]$, the
homogeneous Besov space $\dot{B}^{s}_{p,q}(\mathbb{R}^{n})$ is
defined by
\begin{equation*}
  \dot{B}^{s}_{p,q}(\mathbb{R}^{n})=\big\{f\in \mathcal{Z}'(\mathbb{R}^{n}):\ \
  \|f\|_{\dot{B}^{s}_{p,q}}<\infty\big\},
\end{equation*}
where
\begin{equation*}
  \|f\|_{\dot{B}^{s}_{p,q}}=
\begin{cases}
  \Big(\sum_{j\in\mathbb{Z}}2^{jsq}\|\Delta_{j}f\|_{L^{p}}^{q}\Big)^{1/q}\
  \ \text{for}\ \ 1\leq q<\infty,\\
  \sup_{j\in\mathbb{Z}}2^{js}\|\Delta_{j}f\|_{L^{p}}\ \ \ \ \ \ \ \ \ \text{for}\
  \ q=\infty.
\end{cases}
\end{equation*}
It is well-known that if either $s<\frac{n}{p}$ or $s=\frac{n}{p}$
and $q=1$, then
$(\dot{B}^{s}_{p,q}(\mathbb{R}^{n}),\|\cdot\|_{\dot{B}^{s}_{p,q}})$
is a Banach space. Moreover, if we denote $D^s
f=\mathcal{F}^{-1}(|\xi|^{s}\mathcal{F}(f))$, then for any function
$f$ defined on $\mathbb{R}^{n}\backslash\{0\}$ which is smooth and
homogeneous of degree $k$, the corresponding pseudo-differential
operator $f(D)$ is a bounded linear map from
$\dot{B}^{s}_{p,q}(\mathbb{R}^{n})$ to
$\dot{B}^{s-k}_{p,q}(\mathbb{R}^{n})$. Besides, there exists a
constant $C$ depending only on the dimension $n$ such that for any
$s>0$, $j\in\mathbb{Z}$ and $1\leq p\leq q\leq\infty$, there holds
the following Bernstein inequality:
\begin{equation}\label{eq2.1}
  \supp \widehat{f}\subset\{|\xi|\leq2^{j}\}\ \ \Longrightarrow\ \
  \|D^{s}f\|_{L^{q}}\leq
  C2^{js+jn(1/p-1/q)}\|f\|_{L^{p}}.
\end{equation}
We now recall the definition of the Chemin-Lerner space
$\mathfrak{L}^{r}(0,T; \dot{B}^{s}_{p,q}(\mathbb{R}^{n}))$:

\begin{definition}\label{de2.1} {\em (\cite{C98})} Let $s\in \mathbb{R}$, $1\leq p, q, r\leq\infty$, and
$0<T\leq\infty$ be fixed. The Chemin-Lerner space is defined by
\begin{equation*}
  \mathfrak{L}^{r}(0,T; \dot{B}^{s}_{p,q}(\mathbb{R}^{n})):=\{f\in
  \mathcal{D}'((0,T), \mathcal{Z}'(\mathbb{R}^{n})):\ \
  \|f\|_{\mathfrak{L}^{r}(0,T;
  \dot{B}^{s}_{p,q}(\mathbb{R}^{n}))}<\infty\},
\end{equation*}
where
$
  \|f\|_{\mathfrak{L}^{r}(0,T; \dot{B}^{s}_{p,q})}=\big(\sum_{j\in\mathbb{Z}}2^{jsq}\|\Delta_{j}f\|_{L^{r}(0,T;
  L^{p})}^{q}\big)^{1/q}.
$
\end{definition}

 We define the usual space $L^{r}(0,T;
\dot{B}^{s}_{p,q}(\mathbb{R}^{n}))$ associated with the norm
\begin{equation*}
  \|f\|_{L^{r}(0,T; \dot{B}^{s}_{p,q})}=\Big(\int_{0}^{T}\Big(\sum_{j\in\mathbb{Z}}2^{jsq}\|\Delta_{j}f\|_{
  L^{p}}^{q}\Big)^{r/q}dt\Big)^{1/r}.
\end{equation*}
By the  Minkowski inequality, it is readily to verify that
\begin{equation*}
\begin{cases}
  \|f\|_{\mathfrak{L}^{r}(0,T; \dot{B}^{s}_{p,q})}\leq\|f\|_{L^{r}(0,T;
  \dot{B}^{s}_{p,q})} \ \ \  \text{if}\ \ \  r\leq q,\\
  \|f\|_{L^{r}(0,T; \dot{B}^{s}_{p,q})}\leq \|f\|_{\mathfrak{L}^{r}(0,T;
  \dot{B}^{s}_{p,q})} \ \ \ \text{if} \ \ \  q\leq r.
\end{cases}
\end{equation*}

In our discussion we shall use two basic results related to this
space. The first one is concerned with the product of two functions
in this space and reads as follows:

\begin{lemma}\label{le2.2} {\em (\cite{D05, RS96})}
Let $1\leq p$, $q$, $r$, $r_{1}$, $r_{2}\leq \infty$, $s_{1}$,
$s_{2}<\frac{n}{p}$, $s_{1}+s_{2}>0$ and
$\frac{1}{r}=\frac{1}{r_{1}}+\frac{1}{r_{2}}$. Then there exists a
positive constant $C$ depending only on $s_{1}, s_{2}, p, q, r,
r_{1}, r_{2}$ and $n$ such that
\begin{equation}\label{eq2.2}
  \|fg\|_{\mathfrak{L}^{r}(0,T; \dot{B}^{s_{1}+s_{2}-n/p}_{p,
  q})}\leq C\|f\|_{\mathfrak{L}^{r_{1}}(0,T; \dot{B}^{s_{1}}_{p,
  q})}\|g\|_{\mathfrak{L}^{r_{2}}(0,T; \dot{B}^{s_{2}}_{p,
  q})}.
\end{equation}
\end{lemma}
The proof of this result is a simple application of the following
fundamental result concerning the product of two functions in the
homogeneous Besov spaces: Let $1\leq p,q\leq\infty$,
$s_1,s_2<\frac{n}{p}$ and $s_1+s_2>0$. Then there exists a constant
$C$ depending only on $p, q, s_1, s_2$ and $n$ such that
$$
  \|fg\|_{\dot{B}^{s_1+s_2-n/p}_{p,q}}\leq
  C\|f\|_{\dot{B}^{s_1}_{p,q}}\|g\|_{\dot{B}^{s_2}_{p,q}}.
$$
For details of the proof, we refer the reader to see \cite{D05} and
\cite{RS96}. The second one, whose proof can be found from e.g.
\cite{D05}, is concerned with the Cauchy problem of the heat
equation:
\begin{equation}\label{eq2.3}
\begin{cases}
  \frac{\partial u}{\partial t}-\Delta u= f(x,t), \ \
  x\in\mathbb{R}^{n}, \ t>0,\\
  u(x,0)=u_{0}(x), \ \ x\in\mathbb{R}^{n}.
\end{cases}
\end{equation}

\begin{proposition}\label{pro2.3} {\em (\cite{D05})}
Let $s\in \mathbb{R}$, $1\leq p,q,r_1\leq\infty$ and
$0<T\leq\infty$. Assume that $u_{0}\in
\dot{B}^{s}_{p,q}(\mathbb{R}^{n})$ and $f\in\mathfrak{L}^{r_1}(0,T;
\dot{B}^{s+2/r_1-2}_{p,q}(\mathbb{R}^{n}))$. Then \eqref{eq2.3} has
a unique solution $u\in\underset{r_1\leq
r\leq\infty}{\cap}\mathfrak{L}^{r}(0,T;
\dot{B}^{s+2/r}_{p,q}(\mathbb{R}^{n}))$. In addition, there exists a
constant $C>0$ depending only on $n$ such that for any $r_1\leq
r\leq\infty$, we have
\begin{equation}\label{eq2.4}
  \|u\|_{\mathfrak{L}^{r}(0,T; \dot{B}^{s+2/r}_{p,q})}\leq
  C\big(\|u_{0}\|_{\dot{B}^{s}_{p,q}}+\|f\|_{\mathfrak{L}^{r_1}(0,T;
  \dot{B}^{s+2/r_1-2}_{p,q})}\big).
\end{equation}
\end{proposition}

Next we recall an existence and uniqueness result for an abstract
operator equation in a generic Banach space. For the proof we refer
the reader to see Lemari\'{e}-Rieusset \cite{L02}.

\begin{proposition}\label{pro2.4} {\em (\cite{L02})}
Let $\mathcal{X}$ be a Banach space and
$\mathbf{B}:\mathcal{X}\times\mathcal{X}\rightarrow\mathcal{X}$ is a
bilinear bounded operator, $\|\cdot\|_{\mathcal{X}}$ being the
$\mathcal{X}$-norm. Assume that for any $u_{1},u_{2}\in
\mathcal{X}$, we have $\|\mathbf{B}(u_{1},u_{2})\|_{\mathcal{X}}\leq
C_{0}\|u_{1}\|_{\mathcal{X}}\|u_{2}\|_{\mathcal{X}}$. Then for any
$y\in \mathcal{X}$ such that $\|y\|_{\mathcal{X}}\leq
\varepsilon<\frac{1}{4C_{0}}$, the equation $u=y+\mathbf{B}(u,u)$
has a solution $u$ in $\mathcal{X}$. Moreover, this solution is the
only one such that $\|u\|_{\mathcal{X}}\leq 2\varepsilon$, and
depends continuously on $y$ in the following sense: if
$\|\widetilde{y}\|_{\mathcal{X}}\leq \varepsilon$,
$\widetilde{u}=\widetilde{y}+\mathbf{B}(\widetilde{u},\widetilde{u})$
and $\|\widetilde{u}\|_{\mathcal{X}}\leq 2\varepsilon$, then
$\|u-\widetilde{u}\|_{\mathcal{X}}\leq \frac{1}{1-4\varepsilon
C_{0}}\|y-\widetilde{y}\|_{\mathcal{X}}$.
\end{proposition}

We are now ready to give the proof of Theorem 1.1. Let $p$ and $q$
be as in Theorem 1.1, i.e., $2\leq p<2n$ and $1\leq q\leq\infty$,
and let $r$ be as in \eqref{eq1.4}, i.e., $1<r\leq\infty$. We choose
a number $2<r_1\leq2r$ such that
$\frac{2}{r_1}+\frac{n}{p}>\frac{3}{2}$. For $T>0$ to be specified
later, we set $\mathcal{X}_{T}=\mathfrak{L}^{r_1}(0,T;
\dot{B}^{-2+n/p+2/r_1}_{p,q}(\mathbb{R}^{n}))$. Given $(v,w)\in
\mathcal{X}_{T}$, we define $\mathcal{G}(v,w)=(\bar{v},\bar{w})$ to
be the solution of the following initial value problem:
\begin{align}\label{eq2.5}
  &\partial_{t}\bar{v}-\Delta
  \bar{v}=-\nabla\cdot(v\nabla(-\Delta)^{-1}(w-v)),
  \ \ \ \bar{v}(x,0)=v_{0}(x),\\
\label{eq2.6}
  &\partial_{t}\bar{w}-\Delta
  \bar{w}=\nabla\cdot(w\nabla(-\Delta)^{-1}(w-v)), \ \ \ \bar{w}(x,0)=w_{0}(x).
\end{align}
Obviously, $(v,w)$ is a solution of \eqref{eq1.3} if and only if it
is a fixed point of $\mathcal{G}$.

\begin{lemma}\label{le2.5} Let $(v, w)\in
\mathcal{X}_{T}$. Then  $(\bar{v}, \bar{w})\in \mathcal{X}_{T}$.
Moreover, there exists a constant $C_0>0$ such that
\begin{align}\label{eq2.7}
  &\|\bar{v}\|_{\mathcal{X}_{T}}\leq
  \|e^{t\Delta}v_{0}\|_{\mathcal{X}_{T}}+C_0\|(v,w)\|_{\mathcal{X}_{T}}^{2},\\
\label{eq2.8}
  &\|\bar{w}\|_{\mathcal{X}_{T}}\leq
  \|e^{t\Delta}w_{0}\|_{\mathcal{X}_{T}}+C_0\|(v,w)\|_{\mathcal{X}_{T}}^{2}.
\end{align}
Here $e^{t\Delta}$ is the heat operator with kernel $G(x,t)=(4\pi
t)^{-n/2}\exp(-\frac{|x|^2}{4t})$.
\end{lemma}
\begin{proof}
By Duhamel principle, \eqref{eq2.5} is equivalent to the following
integral equation:
$$
  \bar{v}(t)=e^{t\Delta}v_{0}-\int_{0}^{t}e^{(t-\tau)\Delta}\nabla\cdot(v\nabla(-\Delta)^{-1}(w-v))(\tau)d\tau.
$$
Since $2\leq p<2n$, $2<r_1<\infty$ and
$\frac{n}{p}+\frac{2}{r_1}>\frac{3}{2}$, by choosing
$s_{1}=-2+\frac{n}{p}+\frac{2}{r_1}$ and
$s_{2}=-1+\frac{n}{p}+\frac{2}{r_1}$ in Lemma \ref{le2.2} we get
\begin{align}\label{eq2.9}
  \|\nabla\cdot(&v\nabla(-\Delta)^{-1}(w-v))\|_{\mathfrak{L}^{r_1/2}(0,T;
  \dot{B}^{-4+n/p+4/r_1}_{p,q})}\nonumber\\&\leq C\|v\nabla(-\Delta)^{-1}(w-v)\|_{\mathfrak{L}^{r_1/2}(0,T;
  \dot{B}^{-3+n/p+4/r_1}_{p,q})}\nonumber\\
  &\leq C\|v\|_{\mathfrak{L}^{r_1}(0,T;
  \dot{B}^{-2+n/p+2/r_1}_{p,q})}\|\nabla(-\Delta)^{-1}(w-v)\|_{\mathfrak{L}^{r_1}(0,T;
  \dot{B}^{-1+n/p+2/r_1}_{p,q})}\nonumber\\
  &\leq C\|(v,w)\|_{\mathfrak{L}^{r_1}(0,T;
  \dot{B}^{-2+n/p+2/r_1}_{p,q})}^{2}.
\end{align}
Hence, by using Proposition \ref{pro2.3} we conclude that there
exists a positive constant $C_{0}$ such that
\begin{align}\label{eq2.10}
  \|\bar{v}\|_{\mathcal{X}_{T}}&\leq
  \|e^{t\Delta}v_{0}\|_{\mathcal{X}_{T}}+C\|\nabla\cdot(v\nabla(-\Delta)^{-1}(w-v))\|_{\mathfrak{L}^{r_1/2}(0,T;
  \dot{B}^{-4+n/p+4/r_1}_{p,q})}\nonumber\\
  &\leq
  \|e^{t\Delta}v_{0}\|_{\mathcal{X}_{T}}+C_{0}\|(v,w)\|_{\mathfrak{L}^{r_1}(0,T;
  \dot{B}^{-2+n/p+2/r_1}_{p,q})}^{2}.
\end{align}
This proves \eqref{eq2.7}. The proof of \eqref{eq2.8} is similar.
\end{proof}

{\em Proof of Theorem 1.1}:\ \ The above lemma ensures that
$\mathcal{G}$ is well-defined and maps $\mathcal{X}_{T}$ into
itself. Moreover, from \eqref{eq2.7} and \eqref{eq2.8} we see that
for any $(v,w)\in \mathcal{X}_{T}$ and $(\bar{v},\bar{w})=
\mathcal{G}(v,w)$,
\begin{align}\label{eq2.11}
  \|(\bar{v},\bar{w})\|_{\mathcal{X}_{T}}\leq
  \|(e^{t\Delta}v_{0}, e^{t\Delta}w_{0})\|_{\mathcal{X}_{T}}
  + C_{0}\|(v,w)\|_{\mathcal{X}_{T}}^2.
\end{align}

\textbf{Existence}.\ \ We first prove global existence for small
initial data. For this purpose we choose $T=\infty$. By Proposition
\ref{pro2.4} and \eqref{eq2.11} it is easy to see that if
$\|(e^{t\Delta}v_{0},
e^{t\Delta}w_{0})\|_{\mathcal{X}_{\infty}}\leq\varepsilon$ and
$\varepsilon>0$ is so small that $4C_0\varepsilon\leq 1$, then
$\mathcal{G}$ has a fixed point in the closed ball
$\|(v,w)\|_{\mathcal{X}_{\infty}}\leq 2\varepsilon$ in
$\mathcal{X}_{\infty}$. By Proposition \ref{pro2.3} we see that the
condition $\|(e^{t\Delta}v_{0},
e^{t\Delta}w_{0})\|_{\mathcal{X}_{\infty}}\leq\varepsilon$ is
satisfied if $\|(v_{0}, w_{0})\|_{\dot{B}^{-2+n/p}_{p,q}}$ is small
enough. Indeed, by Proposition \ref{pro2.3}, there exists a positive
constant $C_{1}$ depending only on $n$ such that
$\|(e^{t\Delta}v_{0}, e^{t\Delta}w_{0})\|_{\mathcal{X}_{\infty}}\leq
C_{1}\|(v_{0}, w_{0})\|_{\dot{B}^{-2+n/p}_{p,q}}$. Hence, if we
assume that $\|(v_{0}, w_{0})\|_{\dot{B}^{-2+n/p}_{p,q}}\leq
C_1^{-1}\varepsilon$, then we have $\|(e^{t\Delta}v_{0},
e^{t\Delta}w_{0})\|_{\mathcal{X}_{\infty}}\leq \varepsilon$. This
proves global existence for small initial data.

Next we prove local existence for large initial data. For this
purpose we split $v_{0}$ into a sum as follows:
$\widehat{v}_{0}(\xi)=\widehat{v}_{0}1_{\{|\xi|>2^{N}\}}+\widehat{v}_{0}1_{\{|\xi|\leq2^{N}\}}:
=\widehat{v_{01}}+\widehat{v_{02}}$, where $1_{\mathcal{D}}$
represents the characteristic function on the domain $\mathcal{D}$.
Similarly, we split $w_{0}$ as
$\widehat{w_{0}}=\widehat{w_{01}}+\widehat{w_{02}}$. Since $2\leq
p<2n$, it is easy to see that if $N\in\mathbb{Z}^{+}$ is
sufficiently large then
$C_{1}\|(v_{01},w_{01})\|_{\dot{B}^{-2+n/p}_{p,q}}\leq
\frac{1}{2}\varepsilon$. Choosing a such $N$ and fixing it, we get
\begin{equation}\label{eq2.12}
  \|(e^{t\Delta}v_{0},
  e^{t\Delta}w_{0})\|_{\mathcal{X}_{T}}\leq
  \frac{1}{2}\varepsilon+\|(e^{t\Delta}v_{02},
  e^{t\Delta}w_{02})\|_{\mathcal{X}_{T}}.
\end{equation}
Applying the Bernstein inequality,
\begin{align*}
  \|&(e^{t\Delta}v_{02},
  e^{t\Delta}w_{02})\|_{\mathcal{X}_{T}}=\|(e^{t\Delta}v_{02},
  e^{t\Delta}w_{02})\|_{\mathfrak{L}^{r_1}(0,T;
  \dot{B}^{-2+n/p+2/r_1}_{p,q})}\\
  &\lesssim 2^{(2N)/r_1}\|(e^{t\Delta}v_{02},
  e^{t\Delta}w_{02})\|_{\mathfrak{L}^{r_1}(0,T;
  \dot{B}^{-2+n/p}_{p,q})}\leq C_{2}2^{(2N)/r_1}T^{1/r_1}\|(v_{0},
  w_{0})\|_{\dot{B}^{-2+n/p}_{p,q}}.
\end{align*}
Hence, if we choose $T$ small enough such that
$C_{2}2^{(2N)/r_1}T^{1/r_1}\|(v_{0},w_{0})\|_{\dot{B}^{-2+n/p}_{p,q}}\leq\frac{1}{2}\varepsilon$,
then $\|(e^{t\Delta}v_{02}$,
$e^{t\Delta}w_{02})\|_{\mathcal{X}_{T}}\leq \frac{\varepsilon}{2}$.
This result together with \eqref{eq2.12} yields that
$\|(e^{t\Delta}v_{0},
e^{t\Delta}w_{0})\|_{\mathcal{X}_{T}}\leq\varepsilon$. By applying
Proposition \ref{pro2.4} again, we obtain a fixed point of
$\mathcal{G}$ in the closed ball $\|(v,w)\|_{\mathcal{X}_{T}}\leq
2\varepsilon$ in $\mathcal{X}_{T}$, and concludes the proof of local
existence of solution.

\textbf{Regularity}.\ \ Note that if $(v,w)\in \mathcal{X}_{T}$ is a
solution of \eqref{eq1.1}, then we can proceed in the same way as in
the proof of Lemma \ref{le2.5} to obtain that
$$
  \nabla\cdot(v\nabla(-\Delta)^{-1}(w-v)), \nabla\cdot(w\nabla(-\Delta)^{-1}(w-v))\in  \mathfrak{L}^{r_1/2}(0,T;
  \dot{B}^{-4+n/p+4/r_1}_{p,q}(\mathbb{R}^{n})).
$$
By Proposition \ref{pro2.3}, for any $\frac{r_1}{2}\leq r\leq\infty$
we have $(v, w)\in \mathfrak{L}^{r}(0,T;
\dot{B}^{-2+n/p+2/r}_{p,q}(\mathbb{R}^{n}))$. Moreover, if
$(v_{0},w_{0})$ belongs to the closure of
$\mathcal{S}(\mathbb{R}^n)$ in the space
$\dot{B}^{-2+n/p}_{p,q}(\mathbb{R}^{n})$, then $(v, w)\in C([0,T),
\dot{B}^{-2+n/p}_{p,q}(\mathbb{R}^{n}))$.

\textbf{Uniqueness}.\ \ Let $(v,w)$ and $(\tilde{v},\tilde{w})$ be
two solutions of \eqref{eq1.1} in $\mathcal{X}_{T}$ associated with
the initial data $(v_{0},w_{0})$ and $(\tilde{v}_{0},
\tilde{w}_{0})$, respectively. Set $V=v-\tilde{v}$, $W=w-\tilde{w}$.
Then $(V,W)$ satisfies the following equations:
\begin{equation*}
\begin{cases}
  &\partial_{t}V-\Delta
  V=-\nabla\cdot(V\nabla(-\Delta)^{-1}(w-v))-\nabla\cdot(\tilde{v}\nabla(-\Delta)^{-1}(W-V)),\\
  &\partial_{t}W-\Delta
  W=\nabla\cdot(W\nabla(-\Delta)^{-1}(w-v))+\nabla\cdot(\tilde{w}\nabla(-\Delta)^{-1}(W-V)),\\
  &V(x,0)=V_{0}(x)=v_{0}(x)-\tilde{v}_{0}(x),\
  \ W(x,0)=W_{0}(x)=w_{0}(x)-\tilde{w}_{0}(x).
\end{cases}
\end{equation*}
Proceeding in the same way as in the proof of Lemma \ref{le2.5}, we
can prove that
\begin{align*}
  \|-\nabla\cdot&(V\nabla(-\Delta)^{-1}(w-v))-\nabla\cdot(\tilde{v}\nabla(-\Delta)^{-1}(W-V))\|_{\mathcal{X}_{T}}\\
  &\leq C
  \big(\|v\|_{\mathcal{X}_{T}}+\|\tilde{v}\|_{\mathcal{X}_{T}}+\|w\|_{\mathcal{X}_{T}}\big)
  \|(V,W)\|_{\mathcal{X}_{T}}\\
  \|\nabla\cdot&(W\nabla(-\Delta)^{-1}(w-v))+\nabla\cdot(\tilde{w}\nabla(-\Delta)^{-1}(W-V))\|_{\mathcal{X}_{T}}\\
  &\leq C
  \big(\|v\|_{\mathcal{X}_{T}}+\|w\|_{\mathcal{X}_{T}}+\|\tilde{w}\|_{\mathcal{X}_{T}}\big)
  \|(V,W)\|_{\mathcal{X}_{T}}.
\end{align*}
Hence, by Proposition \ref{pro2.3},
\begin{align*}
  \|(V,W )\|_{\mathcal{X}_{T}}&\leq C_{1}\|(V_{0},
  W_{0})\|_{\dot{B}^{-2+n/p}_{p,q}}\\&+ C_{0}
  \big(\|v\|_{\mathcal{X}_{T}}+\|\tilde{v}\|_{\mathcal{X}_{T}}+\|w\|_{\mathcal{X}_{T}}+\|\tilde{w}\|_{\mathcal{X}_{T}}\big)
  \|(V,W)\|_{\mathcal{X}_{T}}.
\end{align*}
Let
$M(T):=C_{0}\big(\|v\|_{\mathcal{X}_{T}}+\|\tilde{v}\|_{\mathcal{X}_{T}}
+\|w\|_{\mathcal{X}_{T}}+\|\tilde{w}\|_{\mathcal{X}_{T}}\big)$. By
absolute continuity of the Lebesgue integral, we have that $M(T)$
converges to zero as $T\to 0^+$. Hence, if we choose $T_{1}$
sufficiently small such that $M(T_{1})\leq\frac{1}{2}$, then
\begin{equation*}
  \|(V,W )\|_{\mathcal{X}_{T}}\leq 2C_{1}\|(V_{0},
  W_{0})\|_{\dot{B}^{-2+n/p}_{p,q}}.
\end{equation*}
Repeating this argument step by step on the intervals $[0,T_{1})$,
$[T_{1}, 2T_{1})$, $\ldots$, we finally get a constant $C=C_T$ after
a finite steps such that $\|(V,W)\|_{\mathcal{X}_{T}}\leq C\|(V_{0},
W_{0})\|_{\dot{B}^{-2+n/p}_{p,q}}$. This proves \eqref{eq1.5} which
implies the uniqueness of solutions. The proof of Theorem
\ref{th1.1} is complete.

\textbf{Acknowledgment.} This work is supported by the China
National Natural Science Fundation under the Grant No.10771223.


\begin{thebibliography}{00}
\small



\bibitem{BMV04} N. Ben Abdallah, F. M\'{e}hats, N. Vauchelet,
A note on the long time behavior for the drift-diffusion-Poisson
system, C. R. Math. Acad. Sci. Paris, 339 (10) (2004) 683--688.

\bibitem{BD00} P. Biler, J. Dolbeault,
Long time behavior of solutions to Nernst-Planck and
Debye-H\"{u}ckel drift-diffusion systems, Ann. Henri Poincar\'{e}, 1
(2000) 461--472.


\bibitem{BHN94} P. Biler, W. Hebisch, T. Nadzieja,
The Debye system: existence and large time behavior of solutions,
Nonlinear Anal., 23 (1994) 1189--1209.

\bibitem{C98} J.-Y. Chemin,  \textit{Perfect Incompressible Fluids}, Oxford Lecture Series
in Mathematics and its Applications, vol. 14. The Clarendon Press,
Oxford University Press: New York, 1998.

\bibitem{D05} R. Dachin, \textit{Fourier Analysis Methods for PDE's}, 2005,
http://perso-math.univ-mlv.fr/users/danchin.raphael/courschine.pdf.

\bibitem{DH23}  P. Debye,  E. H\"{u}ckel,
Zur Theorie der Elektrolyte, II: Das Grenzgesetz f\"{ur} die
elektrische Leitf\"{a}higkeit, Phys. Z., 24 (1923) 305--325.

\bibitem{G85}  H. Gajewski, On existence, uniqueness and asymptotic behavior of solutions of the basic equations for carrier transport in
semiconductors, Z. Angew. Math. Mech., 65 (1985) 101--108.

\bibitem{GG86} H. Gajewski, K. Gr\"{o}ger,
On the basic equations for carrier transport in semiconductors, J.
Math. Anal. Appl., 113 (1986) 12--35.

\bibitem{K99} G. Karch,
Scaling in nonlinear parabolic equations, J. Math. Anal. Appl., 234
(1999) 534--558.

\bibitem{KO08} M. Kurokiba, T. Ogawa, Well-posedness for the drift-diffusion system in $L^{p}$ arising from the semiconductor device
simulation, J. Math. Anal. Appl., 342 (2008) 1052--1067.

\bibitem{L02} P.-G. Lemari\'{e}-Rieusset,
\textit{Recent Developments in the Navier-Stokes Problem}, Research
Notes in Mathematics, Chapman \& Hall/CRC, 2002.

\bibitem{M74} M. S. Mock, An initial value problem from semiconductor device theory, SIAM J. Math. Anal., 5 (1974) 597--612.

\bibitem{OS08} T. Ogawa, S. Shimizu, The drift-diffusion system in two-dimensional critical Hardy space,
J. Funct. Anal., 255 (2008) 1107--1138.


\bibitem{RS96} T. Runst, W. Sickel, \textit{Sobolev Spaces of Fractional Order, Nemytskij Operators, and Nonlinear Partial Differential
Equations}, de Gruyter Series in Nonlinear Analysis and
Applications, vol. 3. Walter de Gruyter \& Co.: Berlin, 1996.

\bibitem{S83} S. Selberherr,
\textit{Analysis and simulation of semiconductor devices}, Springer
Verlag, 1983.



\end{thebibliography}
\end{document}